\documentclass[12pt]{amsart}
\usepackage{amsmath,amsfonts,amsthm,amsopn,cite,mathrsfs,amssymb}

\usepackage{epsfig,verbatim}
\usepackage{subfigure,color,enumitem}

\numberwithin{equation}{section}

\newcommand{\weakconv}{\xrightarrow{w}}
\newcommand{\WZ}{W_{\mathbb{Z}}}
\newcommand{\W}{W}
\newcommand{\ctilg}{{\sum_{k \in \bZ} c_k g(\cdot-k)}}

\newcommand{\ip}[2]{\ensuremath{\left<#1,#2\right>}}

\newcommand{\otherset}{\Omega}
\newcommand{\sisp}{V}

\newcommand{\tp}{totally positive}
\newcommand{\ltwo}{L^2(\bR )}

\newtheorem*{lemma*}{Lemma}

\DeclareMathOperator*{\supp}{supp}

\DeclareMathOperator*{\essinf}{ess\,inf\,}

\newcommand{\field}[1]{\mathbb{#1}}
\newcommand{\bR}{\field{R}}
\newcommand{\bN}{\field{N}}
\newcommand{\bZ}{\field{Z}}
\newcommand{\bC}{\field{C}}

\newcommand{\bT}{\field{T}}

 \def\cG{\mathcal{G}}

\newcommand{\tfs}{time-frequency shift}

\newcommand{\sis}{shift-invariant space}

\newcommand{\fif}{if and only if}
\def\inv{^{-1}}

\newcommand{\abs}[1]{\lvert#1\rvert}
\newcommand{\norm}[1]{\lVert#1\rVert}

\newcommand{\vs}{\vspace{3 mm}}

\newtheorem{tm}{Theorem}[section]
\newtheorem{lemma}[tm]{Lemma}
\newtheorem{prop}[tm]{Proposition}
\newtheorem{cor}[tm]{Corollary}

\newcommand{\no}{{n_\Lambda}}

\newtheorem{remark}[tm]{Remark}

\begin{document}
\begin{abstract}
We study nonuniform sampling in shift-invariant spaces and the construction of Gabor
frames with respect to the  class of totally positive functions
whose Fourier transform factors as $ \hat g(\xi)= \prod_{j=1}^n (1+2\pi i\delta_j\xi)^{-1} \, e^{-c \xi^2}$ for
$\delta_1,\ldots,\delta_n\in\bR, c >0$ (in which case $g$ is called totally positive of Gaussian type).

In analogy to Beurling's sampling theorem for the Paley-Wiener space of entire functions, we prove that every
separated set with
lower Beurling density  $>1$ is a sampling set for the \sis\ generated by such a $g$.
In view of the known necessary density conditions, this result is optimal and
validates the  heuristic reasonings in the engineering literature.

Using a subtle  connection between sampling in shift-invariant spaces
and the theory of Gabor frames, we show that the set of phase-space
shifts of $g$ with respect to a rectangular lattice $\alpha \bZ \times
\beta \bZ $ forms a frame, if and only if $\alpha \beta <1$. This solves an open problem going back
to Daubechies in 1990 for the class of totally positive functions of
Gaussian type.

The proof strategy involves the connection between sampling in
shift-invariant spaces and Gabor frames, a new characterization of sampling sets ``without inequalities'' in the
style of Beurling, new properties of totally positive functions,
and the interplay between zero sets of functions in a shift-invariant
space and functions in the Bargmann-Fock space.
\end{abstract}

\title[Sampling Theorems for  Shift-invariant Spaces]
{Sampling Theorems for  Shift-invariant Spaces,  Gabor Frames,
  and Totally Positive Functions}
\author{Karlheinz Gr\"ochenig}
\email{karlheinz.groechenig@univie.ac.at}
\address{Faculty of Mathematics \\
University of Vienna \\
Oskar-Morgenstern-Platz 1 \\
A-1090 Vienna, Austria}

\author{Jos\'e Luis Romero}
\email{jose.luis.romero@univie.ac.at}
\email{jlromero@kfs.oeaw.ac.at}
\address{Faculty of Mathematics \\
University of Vienna \\
Oskar-Morgenstern-Platz 1 \\
A-1090 Vienna, Austria}
\address{Acoustics Research Institute, Austrian Academy of Sciences,
Wohllebengasse 12-14 A-1040, Vienna, Austria}

\author{Joachim St\"ockler}
\email{joachim.stoeckler@math.tu-dortmund.de}

\address{Faculty of Mathematics \\
TU Dortmund \\
Vogelpothsweg 87 \\
D-44221 Dortmund, Germany }

\date{}
\thanks{K.\ G.\ was
  supported in part by the  project P26273 - N25  of the
Austrian Science Fund (FWF). J.~L.\ R.\
gratefully acknowledges support from the Austrian Science Fund (FWF):
P 29462 - N35.}
\maketitle

\section{Introduction}

\subsection{Nonuniform sampling}

Beurling's sampling theory for the Paley-Wiener space is at the
crossroad of complex analysis and signal processing and has  served as a model and inspiration for many
generations of sampling theorems in both fields. To recapitulate
Beurling's results,
 let
$$
\mathrm{PW}^2 = \{ f\in \ltwo : \supp \, \hat f \subseteq [-1/2,1/2]\}
$$
be the Paley-Wiener space, where $\hat f(\xi) := \int_\bR f(x) e^{-2 \pi i x \xi} dx$ is the Fourier transform.
It consists of restrictions  of entire functions of
exponential type to the real line and is called  the space of bandlimited functions by
engineers and applied mathematicians. This space  was and still is the
prevailing signal model for phenomena as diverse as radio signals,
temperatures and biomedical measurements \cite{unser00}.

A separated set
$\Lambda \subseteq \bR $ is called a
sampling set for $\mathrm{PW}^2$, if there exist $A,B>0$ such that
\begin{equation}
  \label{eq:ch2}
A \|f\|_2^2 \le \sum_{\lambda\in\Lambda} |f(\lambda)|^2 \le B\|f\|_2^2
\qquad \text{ for all } f\in \mathrm{PW}^2 \, .
\end{equation}
The understanding of sampling sets may be considered a core problem in
complex analysis, but is also a  fundamental problem in signal
processing. Since a signal is always sampled and then processed, sampling is
a key operation in the analog-digital conversion. The sampling
inequality
\eqref{eq:ch2} guarantees that the samples  carry
the complete information about the signal and that the signal can be
recovered from the samples with some stability with respect to
measurement and processing errors.

Beurling ~\cite{be66,be89} characterized sampling sets by means of their density. The (lower) Beurling density is the average density of
samples per unit and is formally defined as
\begin{equation}
  \label{eq:ch3}
     D^-(\Lambda) := \mathop{{\rm liminf}}_{R\to\infty}\inf_{x\in\bR}
  \frac{\# (\Lambda\cap [x-R,x+R])}{2R}.
\end{equation}
The theorems of Beurling give an almost complete
characterization of sampling sets for the Paley-Wiener space
$\mathrm{PW}^2$.
\begin{tm}
\label{th_beur}
(a) If  $\Lambda \subseteq \bR $ is separated and $D^-(\Lambda ) >
1$, then $\Lambda $ is a sampling set  for $\mathrm{PW} ^2$.

(b) Conversely,  if $\Lambda $ is a sampling set for $\mathrm{PW}
^2$, then $D^-(\Lambda
) \geq 1$.
\end{tm}

A related result with uniform densities occurs in the work of Duffin
and Schaeffer \cite{ds52}.   The case of critical density $D^-(\Lambda
)= 1$
is delicate and was settled in
\cite{orse02}.
Subsequently,
the necessary condition (b) above has been extended and adapted to
many different situations, notably to Paley-Wiener spaces in higher
dimension with respect to general spectra by Landau~\cite{landau67},
and to sampling in spaces of analytic functions~\cite{seip03}.
The necessary conditions for sampling can even be formulated and
proved for sampling  in  reproducing kernel Hilbert
spaces over a metric measure space~\cite{bacahela06,FGHKR}.

The sufficient condition (a) is much deeper and more subtle   and has been extended
only to a handful of  situations, notably to sampling in the Bargmann-Fock
space of entire functions by Lyubarskii and Seip
~\cite{lyub92,seip92}. For bandlimited functions with general spectra,
even
the  existence of sets that are both sampling and interpolating  is a
very challenging problem \cite{lev12, koni15}.

In this article we  develop an analog of  Beurling's theory in an unexpected
direction beyond classical complex analysis and study sampling in
shift-invariant spaces. This notion comes  from approximation theory
and  replaces the Paley-Wiener space in many modern applications
because   often  data acquisition demands a more flexible
setting, where
signals are not strictly bandlimited, but only approximately so.  For
the definition, let  $g\in \ltwo $ be a
generating function  and consider the space
\begin{align*}
\sisp^2(g) = \big\{ f\in \ltwo : f = \sum _{k\in \bZ } c_k g(\cdot - k), c\in
\ell ^2(\bZ )\big\}.
\end{align*}
We always assume that $g$ possesses stable translates so that $\|f\|_2
\asymp \|c\|_2$.
 Typically the  Fourier transform of $g$ decays
rapidly outside an interval and thus
functions in $V^2(g)$ are approximately bandlimited. In the extreme
case of $\hat{g} = \chi _{[-1/2,1/2]}$, the shift-invariant space  $\sisp^2(g)$
is precisely the Paley-Wiener space $\mathrm{PW}^2$.
Intuitively, functions in a general \sis\ $\sisp^2(g)$ have a
\emph{rate of innovation} of one degree per space unit \cite{unser00, AG01, blmave02}.  It is therefore expected
that nonuniform sampling theorems similar to Theorem \ref{th_beur}
 hold for general shift-invariant spaces.
Such claims are usually backed by  a heuristic comparison to the bandlimited case,
but, in spite of their central role in signal processing, they have
been given only moderate formal support so
far. In ~\cite{AG00}  we suggested that  every reasonable generator
$g$ comes with a  notion of a suitable $g$-dependent density
$D_g^-(\Lambda )$, such that Beurling's formulation
carries over verbatim. Despite many efforts and dozens of
publications this suggestion and the attached conjectures remained untouched.

The  existing results for sampling  in a \sis\ use the covering
density
$\delta = 2\sup _{x \in \bR} \inf _{\lambda \in \Lambda}
|x-\lambda| $ instead of  the Beurling density.  Typical statements
assert that every set $\Lambda $ with a  ``sufficiently small''
covering density $\delta
$   is a sampling set for $\sisp^2(g)$. However, because of the
unspecified (and usually  large)  gap between the known  necessary density
and the proven   sufficient density
these results are far from optimal and lack practical relevance.
See~\cite{AG01} for a survey of sampling in \sis s  including many qualitative results.
Sharp results in terms of the covering density  (``$\delta < 1$
 is sufficient for sampling") are known only when the generator $g$ is a
B-spline~\cite{AG00}, an exponential B-spline~\cite{kloos-stoeckler},
or a totally positive function of
finite type~\cite{GS13}.

Our objective is to develop a sampling theory in a large class of  \sis s in the style
of Beurling  and to prove the best possible results.
As a natural and suitable class of generators we deal with  totally
positive functions. A function $g \in L^1(\bR )$ is called totally
positive, if for all $N\in \bN$ and real numbers  $x_1 < x_2 < \dots <
x_N$, $y_1< y_2< \dots < y_N$ the matrix
$$
\Big( g(x_j - y_k) \Big) _{j,k= 1, \dots ,N}
$$
has non-negative determinant. This property is fundamental in
statistics, approximation theory,  and  data interpolation~\cite{Karlin1968a}.
In the following we will deal with a subset of totally positive
functions and  call a function $g$ on $\bR $ \tp\ of Gaussian type if its Fourier
transform factors as
\begin{equation}\label{eq:ch1}
  \hat g(\xi)= \prod_{j=1}^n (1+2\pi i\delta_j\xi)^{-1} \, e^{-c \xi
    ^2},\qquad \delta_1,\ldots,\delta_n\in\bR, c >0, n\in \bN \cup \{0\} \, .
\end{equation}
By Schoenberg's characterization~\cite{Schoenb1947a} of \tp\ functions, every $g$
satisfying \eqref{eq:ch1} is  \tp .

Our first main result is the following (almost) characterization of
sampling sets in $V^2(g)$ in analogy to Beurling's theorem.

\begin{tm} \label{tmch1}
Assume that  $g$ is a \tp\ function of  Gaussian type as in \eqref{eq:ch1}.  If $\Lambda
\subseteq \bR $ is separated and
$D^-(\Lambda) > 1$, then $\Lambda $ is a  sampling set for $\sisp^2(g)$.
\end{tm}
Since a sampling set for $\sisp^2(g) $ must necessarily satisfy
$D^-(\Lambda ) \geq 1$ by \cite{AG00,bacahela06}, this result is
optimal.  Theorem \ref{tmch1} thus validates
the heuristic understanding of the signal processing community  that
nonuniform sampling above the critical density (the Nyquist rate) leads to stable
reconstruction in a  shift-invariant space.
As with bandlimited functions \cite{orse02},
the  case of the critical density  $D^-(\Lambda )=1$ is subtle and deserves a more
detailed analysis. Generically,  the set
$\Lambda = r +  \bZ $ is  a sampling set, but there always exists at least one
$r_0 \in [0,1]$ such that $r_0 + \bZ $ fails to be a sampling set.

To the best of our knowledge, Theorem~\ref{tmch1} is
the first result of this kind beyond a purely analytic setting
(although complex variable methods enter our proofs through a
backdoor).  In contrast to the covering density, Theorem \ref{tmch1} allows for
truly nonuniform sets with large gaps  and underscores the use of
Beurling's density as the right performance metric for sampling.

We stress that Theorem~\ref{tmch1} is new
even for the \sis\ with  Gaussian generator.
We believe that it opens a new avenue in approximation theory of radial basis functions, because sampling inequalities
have become an integral part of the error analysis for scattered data interpolation \cite{NWW2005,hamanawa12}.
We also have a dual result for interpolation in \sis s, formulated in terms of
Beurling's upper density $D^+(\Lambda)$.
\begin{tm}
\label{th_intro_interp}
Let $\varphi (x) = e^{-cx^2}$   be a  Gaussian. If $\Lambda \subseteq \bR$ is separated and $D^+(\Lambda) < 1$,
  then $\Lambda $ is  an  interpolating set
for $\sisp^2(\varphi )$. This means that for all $a
\in \ell ^2(\Lambda )$ there exists $f\in \sisp^2(\varphi{} )$ such that
$f(\lambda ) = a_\lambda $.
\end{tm}

\subsection{Gabor analysis}

Surprisingly, the  problem of sampling in \sis s is
intimately connected  with the construction of Gabor frames. As a
consequence of Theorem~\ref{tmch1} we  make substantial progress
towards  a conjecture that was originally expressed by I.\ Daubechies
in~\cite{da90} and then refined in~\cite{GS13,mystery}.

 Let
$ (x,\xi )  \in \bR ^2 $  and $ M_\xi T_x g(t) =
e^{2\pi i \xi t} g(t-x)$ be the corresponding phase-space shift (\tfs
). For  a separated set $\Omega \subseteq \bR ^2$ and a   generator $g\in \ltwo $
  (a window function in the established terminology), let  $\cG (g,
\Omega ) = \{ M_\nu T_\mu  g: (\mu ,\nu )  \in \Omega \}$ be the set of
\tfs s over $\Omega $. We study the question when $\cG (g,\Omega )$ is
a Gabor frame, i.e., when there exist $A,B>0$, such that
\begin{equation} \label{gfra}
  A\|f\|_2^2 \leq \sum _{(\mu ,\nu )  \in \Omega } |\langle f, M_\nu T_\mu
  g\rangle |^2 \leq B \|f\|_2^2 \qquad \text{ for all } f \in \ltwo \,  .
\end{equation}
In other words, $\Omega $ is a sampling set for the space $V = \{F
\in L^2(\bR ^2): F(x,\xi ) = \langle f, M_\xi T_x g\rangle \text{
  for } f\in L^2(\bR )\}$.
Again there is a universal necessary density condition given by the
density theorem for Gabor frames~\cite{heil07}. \emph{If $\cG (g,\Omega )$
  is a frame for $\ltwo$, then $D^-_2(\Omega ) \geq 1$} (where
$D^-_2$ is the Beurling density on $\bR ^2$). Under mild conditions
on the  window $g$, the inequality is strict $D_2^-(\Omega ) >
1$~\cite{bhw95, AFK14,grorro15}.

On the other hand, the
derivation of optimal sufficient conditions for \eqref{gfra} to hold
is an old and difficult problem in analysis and is usually studied
only for lattices $\Omega = \alpha \bZ \times \beta \bZ $, $\alpha ,
\beta >0$ or $\Omega = A\bZ ^2$ with $A \in \mathrm{GL}(2,\bR )$
~\cite{book,da90}.

Over a period of 25 years    a complete  characterization was  obtained
only for the  Gaussian~\cite{lyub92,seip92}, the truncated and the
symmetric  exponential
functions~\cite{janssen96,janssen03a}, the
hyperbolic secant~\cite{JS02}, and for \tp\ functions of
finite type $\geq 2$~\cite{GS13}. The fact that  all these functions are totally
positive  led to the following conjecture: \emph{For a
  continuous totally positive function $g \in L^1(\mathbb{R})$ the set $\cG (g,\alpha
  \bZ \times   \beta \bZ )$ is a frame, \fif\ $\alpha \beta <1$.}

Our second main result proves this conjecture for totally positive
functions of Gaussian type and extends the characterization of Gabor
frames to semi-regular sets of the form $\Omega = \Lambda \times \bZ $
for $\Lambda \subseteq \bR $.

\begin{tm}
\label{th_intro_gab}
Assume that  $g$ is a \tp \ function of  Gaussian type.  Let $\Lambda
\subseteq \bR $ be a separated set.

 Then $\cG (g, \Lambda \times \bZ )$ is a
Gabor frame for $L^2(\bR )$, \fif\   $D^-(\Lambda ) > 1$.
\end{tm}

We formulate the case of rectangular lattices explicitly as a
corollary.

\begin{cor} \label{corch1}
  Assume that  $g$ is a \tp\ function of  Gaussian type.  Then $\cG
  (g, \alpha \bZ \times \beta \bZ ) $ is a frame, \fif\ $\alpha \beta
  <1$.
\end{cor}

What is still missing to prove the full conjecture  is the
class of \tp\ functions of infinite type, where
$  \hat g(\xi)= \prod_{j=1}^\infty  (1+2\pi i\delta_j\xi)^{-1} \,
e^{-c \xi^2}$, $\delta_j\in\bR \setminus \{0\}$, $c \geq 0$.
 It seems tempting to extend Corollary~\ref{corch1}  by  some form of limiting
procedure, but so far we have not succeeded in finding a rigorous argument.

Let us emphasize that  \tp\  functions constitute the only class of  window functions
 for which  a  characterization of  all frame generating
 lattices seems within reach.  For
 other natural classes of window functions, such as the $B$-splines or
 the Hermite functions, only partial results are known ~\cite{GL09}, and all
 explicit conjectures have been disproved~\cite{lemvig} and need a thorough reformulation.

\subsection{Totally positive functions}
As an important tool, we will derive a new
property of \tp\ functions related to zero sets in shift-invariant
spaces.
The following result will  serve as a central technical tool for the sampling results.
\begin{tm} \label{tmvardim}
  Assume that  $g$ is \tp\ of  Gaussian type.  Let $f = \sum _{k\in \bZ }
c_k g(\cdot - k)$  with coefficients $c\in \ell ^\infty (\bZ )$ and let
$N(f) = \{ x \in \bR : f(x) = 0\}$ be the zero set of
$f$.
Then  either $D^{-}(N(f)) \leq 1$ or $f \equiv 0$.
\end{tm}
This statement resembles the classical result on the density of zeros
of an entire function of exponential type~\cite{boas,levin}, but its
proof is surprisingly indirect and makes little use of complex
analysis methods.  Theorem~\ref{tmvardim} can
also be seen in the context of the variation-diminishing property
of totally positive functions~\cite{Schoenb1947a, Schoenb1948}.

\subsection{Methods}
We draw on methods from different
areas and exploit new interplays between them.

\vs

(i) \emph{Structure theory of Gabor frames.} The duality theory of
Gabor frames \cite{ron-shen97} in the formulation of Janssen~\cite{janssen95} relates the
frame property of $\cG (g, \alpha \bZ \times \beta \bZ )$ to a
(uniform) sampling problem in the \sis\ $\sisp^2(g)$ (with shifts by
$\beta \inv \bZ $ instead of $\bZ $). We  extend this connection  to nonuniform
sampling sets and derive  a characterization of semi-regular Gabor
frames, which we formulate in a coarse version as follows.
\begin{lemma*}
  $\cG (g, (-\Lambda ) \times \bZ )$ is a frame for $\ltwo$,
  \fif\ $x+ \Lambda $ is a sampling set for $\sisp^2(g)$  for all $x$ with uniform
  constants.
 \end{lemma*}
In a crucial step, we will  relate a \emph{uniqueness problem} of the form
``$f \equiv 0$ on $\Lambda$ $\Rightarrow$ $f \equiv 0$ on $\bR$'' in a
\sis\  to the   completeness of  a Gabor
system in a suitable function space~(Theorem~\ref{thm_uni_sis}).

\vs

(ii) \emph{Frames and sampling without inequalities.}
We will derive several new abstract characterizations for sampling sets in \sis s that do not
require an inequality (Theorem~\ref{thm:linfinity}). These characterizations are inspired by
Beurling's approach to  band\-limited functions and deal with the
subtle interplay between sampling on the Hilbert space $\sisp^2(g)$  and
uniqueness on a larger space associated with the $L^\infty$-norm.
While in the Paley-Wiener space there is an explicit bandwidth parameter that
permits such arguments, no direct analog is available for shift-invariant spaces.
To circumvent this obstruction, we use a non-commutative version of
Wiener's Lemma  and the spectral invariance of matrices with $\ell^1$-off-diagonal decay.
 The  depth of the characterization of
sampling sets without inequality is hidden in these results about
spectral invariance, which go back to Baskakov~\cite{bas90} and
Sj\"ostrand~\cite{sj95}.

\smallskip

(iii) \emph{Totally positive functions.} We will use Schoenberg's
characterization of \tp\ functions by their Fourier-Laplace
transform \cite{Schoenb1947a}, but not the total positivity itself.

\vs

It is worthwhile to compare the methods and results for
totally positive functions of finite type in ~\cite{GS13} (where  $\hat g(\xi ) = \prod
_{j=1}^n (1 + 2\pi i \delta _j \xi )\inv $) to those for totally
positive functions of Gaussian type. For totally positive functions of
finite type we made essential use of the total positivity and the
Schoenberg-Whitney conditions to reduce the Gabor frame property to a
finite-dimensional problem. Corollary~\ref{corch1} remains  true, whereas
the statements of Theorems~\ref{tmch1}, \ref{th_intro_gab} and
~\ref{tmvardim}  as stated are false.

Therefore we developed a completely new proof strategy for our
results on totally positive functions of Gaussian type. The line of
arguments is quite  nonlinear and switches  between
sampling in \sis s and Gabor frames. We first prove
Theorem~\ref{tmch1} for the special case
of a Gaussian generator  $\varphi (x) = e^{-cx^2}$.  For this we will invoke the
fundamental characterization of
sampling sets in Bargmann-Fock space by Lyubarskii and
Seip~\cite{lyub92,seip92}, and in fact, slightly elaborate on their main
argument that relates density of zeros and growth. We then show that
$V^2(g)$ for totally positive $g$ of Gaussian type is mapped onto
$V^2(\varphi )$ by a suitable differential operator. This fact allows
us to  compare  the density of the zero sets of functions in the
corresponding shift-invariant spaces. The general version of
Theorem~\ref{tmch1} then follows  from the characterization of
sampling sets without inequalities. Finally we translate the sampling
results for the \sis s $V^2(g)$ into the characterization of Gabor
frames of Theorem~\ref{th_intro_gab}.

\vs

The paper is organized as follows. Section \ref{sec_samp1} introduces
shift-invariant spaces and the connections to
Gabor analysis. Section \ref{sec_wl} develops a characterization of
sampling sets without inequalities  in the spirit of Beurling.
These tools are exploited in Section \ref{sec_samp_gauss} to obtain sampling and interpolation results in
shift-invariant spaces with a Gaussian generator and some necessary
technical extensions.
 The density of zero sets for functions in
shift-invariant spaces with different generators is compared   in
 Section \ref{sec_samp_tp},
leading to the  main results on sampling
(Theorem~\ref{tmch1}). Finally, Section \ref{sec_gab} contains the
characterization of semi-regular Gabor frames
(Theorem~\ref{th_intro_gab}). In the Appendix we supply the postponed
proofs.

\section{Sampling in shift-invariant spaces and nonuniform Gabor families}
\label{sec_samp1}

We first set up the basics about \sis s, formulate several
structural characterizations and point out some connections
between shift-invariant spaces and Gabor families. For future reference we remark that the results in this section
 hold equally for sampling in \sis s in higher
dimensions (with identical proofs).

\subsection{Shift-invariant spaces}

Let $g:\bR \to \bC$ be a function in the Wiener amalgam space
$W_0=W(C,\ell_1)$, i.e., $g$ is continuous and
\begin{align*}
     \|g\|_{W} := \sum_{k\in\bZ} \max_{x\in [k,k+1]} |g(x)| <\infty.
\end{align*}
In the center of our study are the  shift-invariant spaces
\begin{align*}
    \sisp ^p(g) = \left\{ \sum_{k\in\bZ} c_k g(\cdot-k): (c_k)_{k\in\bZ} \in \ell^p(\bZ)\right\}
\end{align*}
for $1\le p\le \infty$. A comprehensive survey on shift-invariant spaces was given by Ron in \cite{Ron2001}.

Standard computations using H\"older's inequality ~\cite{AG01} show that
$$
   \big\| \ctilg \big\|_p\le \|g\|_W ~\|c\|_p,\qquad c\in \ell^p(\bZ),
$$
and consequently $\sisp ^p(g)\subseteq L^p(\bR)$.

Throughout we will assume and use the following stability of the
generator $g$.

\begin{tm}\label{thm:JiaMicchelli}
 Let $g\in W_0$. Then the following are equivalent.
\begin{itemize}
\item[(a)] There exists $p\in [1,\infty ]$ and  $C_p>0$ such that
$$
   \big\| \ctilg\big\|_p\ge C_p\|c\|_p,\qquad c\in \ell^p(\bZ).
$$
\item[(b)] For all  $p\in [1,\infty ]$ there exists a constant   $C_p>0$ such that
$$
   \big\| \ctilg\big\|_p\ge C_p\|c\|_p,\qquad c\in \ell^p(\bZ).
$$
\item[(c)] The shifts $\{g(\cdot-k): k\in\bZ\}$ are
  $\ell^\infty$-independent; this means that \\
$\ctilg\ne 0$ for every $c\in\ell^\infty(\bZ) \setminus\{0\}$.
\item[(d)] For every $\xi\in\bR$ we have $ \sum_{k\in\bZ} |\hat g(\xi+k)|^2 >0.$
\end{itemize}
\end{tm}
We say that $g$ has \emph{stable integer shifts} if any of the equivalent conditions of Theorem \ref{thm:JiaMicchelli}
holds. The equivalence was proved in
\cite[Theorem~3.5]{JiaMicchelli1991} and \cite[Theorem~29]{Ron2001}
under a slightly more general condition on $g$.

\subsection{Sampling, uniqueness, and interpolating sets for shift-invariant spaces}

Since for $g\in W_0$ all spaces $\sisp ^p(g)$ are subspaces of $C(\bR)$, sampling of functions $f\in \sisp ^p(g)$ is well-defined.
A subset  $\Lambda\subseteq \bR$ is called \emph{relatively separated}, if
$$  \no :=  \max _{x\in \bR } \# (\Lambda \cap [x,x+1]) < \infty \, ,$$ and it is called \emph{separated}
if $\inf\{ \abs{\lambda-\mu}: \lambda, \mu \in \Lambda, \lambda \not=
\mu\} = \delta >0$.
If follows that  every relatively separated set can be
partitioned into $n_\Lambda$ separated subsets $\Lambda _j$, $j = 1, \dots, \no$, with separation constant $1/2$.

Clearly, for a  relatively separated set  $\Lambda $
the lower Beurling density
$$
   D^-(\Lambda) := \mathop{{\rm liminf}}_{R\to\infty}\inf_{x\in\bR}
  \frac{\# (\Lambda\cap [x-R,x+R])}{2R}
$$
and the upper Beurling density $D^+(\Lambda)$ (with $\limsup$ and $\sup$ in its definition) are
finite.

 Let $g\in W_0$ have stable integer shifts. A relatively separated
 set  $\Lambda\subseteq \bR$ is called a  {\em sampling set} for the shift-invariant space $\sisp ^p(g)$,
$1\le p<\infty$, if there
are constants $A_p,B_p>0$ such that
\begin{align}
\label{eq_p_samp}
A_p \|f\|_p^p \le \sum_{\lambda\in\Lambda} |f(\lambda)|^p \le B_p\|f\|_p^p
\end{align}
for all $f\in \sisp ^p(g)$.
For $p=\infty$, we require that
\begin{align*}
A_\infty \|f\|_\infty \le \sup_{\lambda\in\Lambda} |f(\lambda)| \le B_\infty \|f\|_\infty
\end{align*}
holds for all $f\in \sisp^\infty(g)$.

A set $\Lambda $ is called a \emph{uniqueness set} for $\sisp ^p(g)$, if $f \in
\sisp ^p(g)$ and $f(\lambda ) = 0$ for all $\lambda \in \Lambda $ implies
that $f = 0$.

If for every $a \in \ell ^p(\Lambda )$ there exists an $f \in V^p(g)$
such that $f(\lambda ) = a_\lambda $, then $\Lambda $ is said to be an {\em interpolating set}.

These notions  can be expressed in terms of the matrices (called the
pre-Gramian in ~\cite{ron-shen97})
\begin{align}
\label{eq_P_matrix}
P_\Lambda(g):= \left( g(\lambda-k)\right) _{\lambda\in\Lambda,~k\in\bZ}.
\end{align}
If  $g\in W_0$, then  by Schur's test, the matrix
$P_\Lambda(g)$ defines a bounded linear operator $P_\Lambda(g):\ell^p(\bZ) \to \ell^p(\Lambda), 1 \leq p \leq \infty$,
and,
\begin{align*}
   \left\| P_\Lambda(g) c\right\|_p \le \no  ~\|g\|_W ~\|c\|_p,\qquad c\in\ell^p(\bZ),
\end{align*}
see for example ~\cite{AG01}. The equivalence of norms
$ \|\sum_k c_k g(\cdot-k)\|_p \asymp \|c\|_p$ and the definitions
imply the following characterizations.

\begin{lemma}
\label{lemma_samp_ops}
Assume that  $g\in W _0 $ has stable integer shifts. Let $1\leq
p\leq \infty $ and  $\Lambda
\subseteq \bR $ be relatively separated. Then:

(a) $\Lambda$ is a sampling set for $\sisp ^p(g)$, if and only
 if $P_\Lambda(g)$ is $p$-bounded below, i.e., if and only if  there exist $\tilde A_p>0$ such that
\begin{align*}
   \left\| P_\Lambda(g) c\right\|_p \ge \tilde A_p \|c\|_p  \quad\text{for all}\quad c\in \ell^p(\bZ).
\end{align*}

(b) $\Lambda $ is a uniqueness set
for $\sisp ^p(g)$,  if and only
 if $P_\Lambda(g)$ is one-to-one on $\ell ^p(\bZ )$.

(c) $\Lambda $ is an interpolating  set for $\sisp ^p(g)$, \fif\
 $P_\Lambda(g)$ is surjective from  $\ell ^p(\bZ )$ onto $\ell
 ^p(\Lambda )$.
\end{lemma}

\subsection{The connection between sampling and Gabor analysis}
Sampling and interpolation problems in the \sis\ $\sisp^2(g)$ are closely related to the construction
of Gabor frames and Riesz bases. The connection is implicit in the
duality theory by Janssen~\cite{janssen95} and  Ron and
Shen~\cite{ron-shen97} and has been instrumental in the
characterization of Gabor frames on lattices for totally positive functions of
finite type~\cite{GS13}. The main objective of this section is to extend the relation
between Gabor analysis and \sis\, to (i) more nonuniform
phase-space nodes and to (ii) uniqueness / completeness problems in
appropriate  function spaces.

We use $T_y g(x)=g(x-y)$ for translation and $M_\xi g(x)=e^{2\pi ix\xi}g(x)$ for modulation.
A  Gabor family is a collection of \tfs s, and we are particularly
interested in the semi-regular family
\begin{align*}
   \cG(g,(-\Lambda) \times \bZ)=
   \{ M_k T_{-\lambda} g: k\in\bZ,~\lambda\in\Lambda\},
\end{align*}
associated with the window function  $g$ and a set $\Lambda \subseteq \bR $.
By definition, this family is a Gabor frame, if there exist constants
$A,B>0$,  such that
\begin{equation}
  \label{eq:gablambdaframe}
  A\|f\|_2^2 \leq  \sum _{k \in \bZ }\sum_{\lambda\in\Lambda}  \left| \langle f, M_k T_{-\lambda} g\rangle\right|^2
  \leq B \|f\|_2^2 \qquad \text{ for all } f  \in L^2(\bR ).
\end{equation}

We next formulate several characterizations of Gabor frames that
establish the connection to sampling in shift-invariant spaces.
We recall the definition of the Zak transform  $Zg(x,\xi)= \sum_{k\in\bZ} g(x+k)e^{2\pi i
  k\xi}$ of a function  $g\in L^1(\bR )$.

\begin{tm}\label{thm:RSLambda}
Assume that  $g\in W_0$ has  stable integer shifts and that
$\Lambda\subseteq \bR$ is relatively separated.
The following are equivalent.
\begin{itemize}
\item[(a)]
The family $\cG(g,(-\Lambda) \times \bZ)$  is a frame for $L^2(\bR)$.

\item[(b)] The set $\Lambda+x$ is a sampling set of $\sisp^2(g)$ for
  every $x\in [0,1)$,  i.e., for each $x \in [0,1)$, there
exist $A_x,B_x>0$
such that
\begin{equation*}
A_x\|c\|_2^2 \leq \sum_{\lambda \in \Lambda}
\Big| \sum_{k \in \bZ} c_k g(\lambda+x-k) \Big|^2
  \leq B_x \|c\|_2^2
  \qquad \mbox{for all } c  \in \ell^2(\bZ ).
\end{equation*}

\item[(c)]  There exist $A,B>0$ such that, for all $ x\in [0,1)$,
\begin{equation}
\label{eq:samplingset2}
A\|c\|_2^2 \leq \sum_{\lambda \in \Lambda}
\Big| \sum_{k \in \bZ} c_k g(\lambda+x-k) \Big|^2
  \leq B \|c\|_2^2
    \qquad\mbox{for all } c  \in \ell^2(\bZ ).
\end{equation}

\item[(d)] There exist constants $A,B>0$ such that
\begin{equation}
  \label{eq:gablambda1}
  A\|h\|_2^2 \leq \sum _{\lambda \in \Lambda } \Big| \int _{0 }^1
    {h(\xi)} Zg(x+\lambda,\xi)\,d\xi
  \Big|^2   \leq B \|h\|_2^2
\end{equation}
for all $h  \in L^2([0,1])$ and $x\in [0,1)$.
\end{itemize}
Moreover, the optimal
constants in \eqref{eq:gablambdaframe}, \eqref{eq:samplingset2}, and \eqref{eq:gablambda1} coincide.
\end{tm}
The proof is  postponed to the Appendix.
\begin{remark}
\normalfont{
\mbox{}
(i) For the case $\Lambda = \alpha \bZ $, the equivalence of (a),
(c), and   (d) is Janssen's version of the duality theory~\cite{janssen95}. It
is one of the most useful tools for deriving concrete and strong
results about Gabor frames~\cite{GS13,janssen96}. The extension of the equivalences
to nonuniform sampling sets is new and is derived by a suitable
modification of the known proofs.

(ii) Condition (d) can be reformulated by saying that for every $x\in [0,1)$
 the set of functions $\{ Zg(x+\lambda ,\xi ): \lambda \in \Lambda \}$
 is a frame for $L^2([0,1], d\xi)$.}
\end{remark}

Interpolation problems in \sis s can similarly be related to Gabor systems. We state the
following theorem, whose proof is given in the Appendix, together with the one of Theorem \ref{thm:RSLambda}.
\begin{tm}
\label{th_interp}
Let $g\in W_0$ have stable integer shifts and $\Lambda \subseteq \bR $ be relatively
separated. Then the following are equivalent.

\begin{itemize}
\item[(a)] $\cG (g, (-\Lambda) \times \bZ ) $ is a Riesz sequence in $L^2(\bR
)$.

\item[(b)] For each $x \in [0,1)$, $\Lambda+x$ is an interpolating set for $\sisp^2(g)$.

\item[(c)] The sets $\Lambda+x$, $x \in [0,1)$, are interpolating sets for $\sisp^2(g)$ with uniform constants,
i.e., for all $a \in \ell ^2(\Lambda)$ there exists $f_x \in \sisp^2(g)$
such that $f_x(\lambda+x)=a_\lambda$, for all $\lambda \in \Lambda$, and $\norm{f_x}_2 \leq C \|a\|_2$,
  with a constant $C $ independent of $x$.
\end{itemize}
\end{tm}

The following theorem relates Gabor systems and \sis s, but this time it concerns
$L^\infty$-uniqueness. It serves as one of our main tools in the following sections.

\begin{tm}
\label{thm_uni_sis}
Let $g \in W_0(\bR)$ and let $\Lambda \subseteq \bR$ be arbitrary. Suppose that
$\cG(g, (-\Lambda) \times \bZ)$ spans a  dense subspace  in $W_0(\bR)$. Then
for every $x \in \bR$,  $\Lambda +x$ is a uniqueness set for $V^\infty
(g)$.
\end{tm}
\begin{proof}
Suppose that $\mathrm{span}\, \cG(g, (-\Lambda) \times \bZ)$ is dense in $W_0(\bR)$. For
simplicity, we assume that $x=0$ (otherwise, we shift $\Lambda$
suitably).  We need to show that  the matrix $P_\Lambda (g) $ is
one-to-one from  $\ell ^\infty (\bZ )$ to $\ell ^\infty (\Lambda
)$. By duality, this is equivalent to saying that the matrix
$P_\Lambda(g)^\prime = (g(\lambda -k))_{k\in \bZ ,  \lambda \in \Lambda }$
viewed as an operator from $\ell ^1(\Lambda ) \to \ell ^1(\bZ )$ has
dense range in $\ell ^1(\bZ )$. Thus
we need to show
that for every  $c \in \ell^1(\bZ)$ and $\epsilon>0$
there is a sequence $a \in \ell^1(\Lambda)$ such that $\|P_\Lambda(g)^\prime
a-c\|_1 <\epsilon $.
Let $\eta \in C^\infty(\bR)$ be supported on $[-1/4,1/4]$
and $\eta (x) =  1$ in a neighborhood of  the origin, and let
\begin{align*}
h=\sum_{k \in \bZ} c_k \eta(\cdot+k).
\end{align*}
Then $h \in W_0$ and $h(-k)=c_k$  for all $k\in\bZ$.

Given $\epsilon>0$,
by assumption, there exists a finite linear combination
$$f(x):=\sum_{\lambda \in \Lambda} \sum_{j \in \bZ} b_{\lambda,j}
e^{2\pi i j x}  g(x+\lambda) $$
with coefficients $b \in \ell^1(\Lambda \times \bZ)$ (with finite
support),  such that
$\norm{f-h}_{W _0} < \epsilon$.

Let $a_\lambda := \sum_j b_{\lambda,j}$, and note that $\norm{a}_{\ell^1(\Lambda)} \leq
\norm{b}_{\ell^1(\Lambda \times \bZ)} <\infty$. In addition,
$f(-k) =  \sum _{\lambda \in \Lambda } \sum _{j\in \bZ } b_{\lambda ,j} g(\lambda -k) = (P_\Lambda(g)^\prime a)_k$.
Finally, we estimate
\begin{align*}
\norm{c-P_\Lambda(g)^\prime a}_{\ell^1(\bZ)}=
\norm{(f-h)| _{\bZ}}_{\ell^1(\bZ)} \leq \norm{f-h}_{W_0} <  \varepsilon.
\end{align*}
Hence the desired approximation holds.
\end{proof}

\section{Characterization of Beurling type for sampling sets and Gabor frames}
\label{sec_wl}

\subsection{Characterization of sampling sets for shift-invariant spaces}
According to \eqref{eq_p_samp},  a sampling set for
$\sisp^p(g)$ is also a uniqueness set for
$\sisp^p(g)$.  The sampling inequalities in
\eqref{eq_p_samp} are stronger than mere uniqueness: they quantify the
stability of the map $f|\Lambda \mapsto f$. It is a remarkable insight
due to Beurling \cite{be66},\cite[p.351-365]{be89}  that sampling sets
for bandlimited functions can be characterized in terms of uniqueness
sets for a larger space of functions. We now develop an instance of
this principle in the context of shift-invariant spaces and  relate  a
sampling problem on $\sisp^2(g)$ to a collection of uniqueness
problems on $\sisp^\infty(g)$. Analogous results hold for
Gabor frames~\cite{gr07-2,grorro15}.

For the formulation, we recall
Beurling's notion of a weak limit of a sequence of  sets.
A sequence $\{\Lambda_n: n \geq 1\}$ of subsets of $\bR$ is said to \emph{converge weakly}
to a set $\Lambda \subseteq \bR$,
denoted $\Lambda_n \weakconv \Lambda$, if for every open bounded interval $(a,b)$ and
every $\varepsilon >0$, there exist $n_0 \in \bN$ such that for all $n \geq n_0$:
\begin{align*}
\Lambda_n \cap (a,b) \subseteq \Lambda + (-\varepsilon,\varepsilon)
\mbox { and }
\Lambda \cap (a,b) \subseteq \Lambda_n + (-\varepsilon,\varepsilon).
\end{align*}
The notion of weak convergence of sets is closely related to local convergence in the Hausdorff metric,
but the precise relation is slightly technical (see \cite[Section 4]{grorro15}.

Given a set $\Lambda \subseteq \bR$, we let  $\WZ(\Lambda)$ denote the
collection of all  weak limits of integer
translates of $\Lambda$, i.e., $\Gamma \in \WZ(\Lambda)$, if there exists a sequence $\{k_n: n \geq 1\} \subseteq \bZ$
such that $\Lambda + k_n \weakconv \Gamma$. If $\Lambda$ is relatively separated, each set in $\WZ(\Lambda)$ is also
relatively separated (see, e.g., \cite[Lemma 4.5]{grorro15}).

The following characterization of sampling sets for $\sisp ^p(g)$  is new and
should be compared with Theorem \ref{thm:JiaMicchelli}.

\begin{tm}\label{thm:linfinity}
Let $g\in W_0$ have stable integer shifts and let
$\Lambda\subseteq \bR$ be a relatively separated set.
The following are equivalent:
\begin{itemize}
\item[(a)] $\Lambda$ is a sampling set of $\sisp ^p(g)$ for some $p \in [1,\infty]$.

\item[(b)] $\Lambda$ is a sampling set of $\sisp ^p(g)$ for all $p \in [1,\infty]$.

\item[(c)] Every $\Gamma \in \WZ(\Lambda)$ is a sampling set for $\sisp^\infty(g)$.

\item[(d)] Every $\Gamma \in \WZ(\Lambda)$ is a
    uniqueness set for $\sisp^\infty(g)$.

\item[(e)] For every $\Gamma \in \WZ(\Lambda)$,
the matrix $P_\Gamma (g)=(g(\gamma-k))_{\gamma \in \Gamma,~k\in\bZ}$
defines a bounded injective operator from  $\ell^\infty(\bZ)$ to $\ell
^\infty (\Gamma )$.
\item[(f)] For every $\Gamma \in \WZ(\Lambda)$,
the matrix  $P_{\Gamma}(g)^\prime =(g(\gamma-k))_{k\in\bZ, \gamma \in \Gamma}$ defines an operator
$\ell^1(\Gamma) \to \ell^1(\bZ)$ with dense range.
\end{itemize}
\end{tm}

The proof of Theorem \ref{thm:linfinity} relies on the theory of
localized frames and combines facts about spectral invariance with
manipulations with weak limits. As we have proved a similar
characterization without inequalities for  Gabor frames in ~\cite[Section 5]{grorro15}, we postpone the proof to  the Appendix.
\begin{remark}
{\normalfont
\mbox{}

(i) Theorem~\ref{thm:linfinity} says that instead of proving an inequality of
type~\eqref{eq:ch2} or the left-invertibility of the matrix $P_\Lambda (g)$,
it suffices to verify that $P_\Gamma (g)$ has trivial kernel on the
larger space $\ell ^\infty (\bZ )$ for every $\Gamma\in W_\bZ(\Lambda)$. This is
conceptually easier.

(ii)  The $\ell^\infty$-injectivity in
(d) cannot be replaced by $\ell^2$-injectivity.
This is in accordance with Theorem \ref{thm:JiaMicchelli} where $\ell^p$-injectivity of $c\mapsto \ctilg$
is a much weaker condition for $p<\infty$. Also note that the
equivalence can hold only for $g\in W_0$, otherwise
some of the conditions do not even make sense.

(iii) We emphasize that  Theorem \ref{thm:linfinity} does not
  say that every uniqueness
set for $\sisp^\infty(g)$ is a sampling set. See Remark
\ref{rem_non_sep} below for a concrete example.

(iv) The exact analog of Theorem \ref{thm:linfinity} is not valid for the Paley-Wiener space, because the sampling problems associated with different $L^p$-norms are not equivalent. In our case, the assumption $g \in W_0$ allows us to invoke a matrix variant of Wiener's $1/f$ lemma due to Sj\"ostrand \cite{sj95}.
}
\end{remark}

\subsection{Characterization of Gabor frames}
Although not needed for the main results, we formulate a new characterization of Gabor frames.
In the formulation we again need weak limits of sequences of sets.
For a set $\Lambda \subseteq \bR$, we define  $\W(\Lambda)$ as the set
of weak limits of translates $\Lambda +x$, i.e., $\Gamma \in W(\Lambda
)$, \fif\ there exists a
sequence of real numbers $\{ x_n : n\geq 1\} \subseteq \bR $, such
that $\Lambda + x_n \weakconv \Gamma$. We note that
$W(\Lambda ) = \bigcup_{x \in [0,1]} \WZ(\Lambda+x)$.

\begin{tm}\label{thm:RSLambdaNEW}
Assume that  $g\in W_0$ has  stable integer shifts and that
$\Lambda\subseteq \bR$ is  relatively separated.
The following are equivalent.
\begin{itemize}
\item[(a)]
The family $\cG(g,(-\Lambda) \times \bZ)$  is a frame for $L^2(\bR)$.

\item[(b)]
 For every $\Gamma \in \W(\Lambda)$, the matrix
$P_\Gamma (g)=(g(\gamma-k))_{\gamma \in \Gamma,~k\in\bZ}$
defines a bounded injective operator from $\ell^\infty(\bZ)$
to $\ell^\infty(\Gamma)$.
\item[(c)] For every $\Gamma \in \W(\Lambda)$, the matrix
$P_\Gamma(g) ^\prime=(g(\gamma-k))_{k\in\bZ,~\gamma \in \Gamma}$
defines an operator  $\ell^1(\Gamma)\to \ell^1(\bZ)$ with dense range.
\end{itemize}
\end{tm}
 Theorem~\ref{thm:RSLambdaNEW} is a direct consequence of Theorems
 \ref{thm:RSLambda} and \ref{thm:linfinity}.

\section{Nonuniform sampling and Gabor families with Gaussian generators}
\label{sec_samp_gauss}

Let $\varphi(x)=e^{-c x^2}$, $c>0$, be a scaled centered Gaussian.
Lyubarskii~\cite{lyub92} and Seip \cite{seip92} proved the following.

\begin{tm}[Lyubarskii, Seip]
\label{th_ly_seip}
Let $\otherset \subseteq \bR^2$ be separated. Then

(a) $\cG(\varphi,\otherset):= \{M_{\xi_2} T_{\xi_1} \varphi : (\xi_1,\xi_2)\in \otherset\}$
is a frame for $L^2(\bR)$ if and only if
\begin{align*}
  D_2^-(\otherset) := \mathop{{\rm liminf}}_{R\to\infty}\inf_{z\in\bR^2}
  \frac{\# (\otherset\cap (z + [-R,R]^2))}{4R^2} >1,
\end{align*}
and

(b) $\cG(\varphi,\otherset)$ is a Riesz sequence if and only if
\begin{align*}
  D_2^+(\otherset) := \mathop{{\rm limsup}}_{R\to\infty}\sup_{z\in\bR^2}
  \frac{\# (\otherset\cap (z + [-R,R]^2))}{4R^2} <1.
\end{align*}
\end{tm}
\begin{remark}
\label{rem_non_sep}

\normalfont{
While a  set $\Omega $  generating a Gabor  Riesz sequence is necessarily separated,
a set generating a  frame may be only relatively separated, and  Theorem \ref{th_ly_seip} (a)
 does not directly apply to non-separated sets. The exact
characterization is the following: $\cG(\varphi ,\otherset)$ is a frame
of $L^2(\bR)$ if and only if $\otherset$ is relatively separated and  contains a separated subset of density
$>1$.

To see the relevance of this subtle distinction,  we  consider the set
$\otherset = \{ (k,j), (k+2^{-|k|},j): k,j \in \bZ\}$. Then clearly
$\Omega $ is relatively separated and
$D^{-}(\otherset) =2$,  but $\cG(\varphi ,\otherset)$ is \emph{not} a frame of
$L^2(\bR)$. Indeed,
if $\cG(\varphi ,\otherset)$ were  a frame of $L^2(\bR)$, then by \cite[Theorem
5.1]{grorro15} (equivalence of (i) and  (v)) $\cG(\varphi ,\otherset')$ would be a frame  for every
$\otherset ' \in
\W(\otherset)$. However,
it is easy to see that $\bZ^2 \in \W(\otherset)$ and it is well-known
that $\cG(\varphi ,\bZ^2)$ is not a frame~\cite{bhw95,heil07}.
}
\end{remark}

In spite of the previous remark, we have the following result.
\begin{tm}
\label{tm_uniq}
Let $\otherset \subseteq \bR^2$ be an arbitrary set with $D^{-}(\otherset)>1$. Then
$\cG(\varphi ,\otherset)$ spans a dense subspace of $W_0(\bR)$.
\end{tm}
The proof of Theorem \ref{tm_uniq} uses (i) the connection between Gabor
systems with a Gaussian generator and sampling in the Bargmann-Fock
space of entire functions, provided by the Bargmann transform \cite{ba61, dagr88},
and (ii) a complex analysis argument that compares growth and density
of zeros
that goes back to Beurling \cite{be66,be89}, and in the form required here was given by
Brekke and Seip~\cite{BS93}. Since the claim that the result is valid for possibly non-separated sets is an essential technical point not explicitly stated in the references, we sketch a
self-contained argument in the Appendix.

We have now assembled all tools to prove a general result for the shift-invariant spaces $V^p(\varphi)$ in the style of
Beurling. As a consequence of Theorem \ref{thm:RSLambda}, we obtain the following statement, which implies
Theorem \ref{th_intro_interp} (announced in the Introduction).

\begin{tm}\label{cor:Gauss}
 Let $\varphi(x)=e^{-cx^2}$  with $c>0$.
\begin{itemize}
 \item[(a)] If  $\Lambda\subseteq \bR$ has density  $D^-(\Lambda)>1$, then
 $\Lambda$ is a uniqueness set for $\sisp^\infty(\varphi )$.
\item[(b)] If $\Lambda$ is separated and $D^-(\Lambda)>1$, then $\Lambda$ is a sampling
set for $\sisp^p(\varphi)$ for $1\leq p\leq \infty $.
\item[(c)] If $\Lambda$ is separated and $D^+(\Lambda)<1$, then $\Lambda$ is an interpolating
set for $\sisp^2(\varphi)$.
\end{itemize}
\end{tm}

\begin{proof}
(a) Assume that $D^-(\Lambda)>1$. Then
the set $\otherset=(-\Lambda)\times \bZ\subseteq \bR^2$ has lower
Beurling density
\begin{align*}
    D_2^-((-\Lambda)\times \bZ) = D^-(-\Lambda) = D^-(\Lambda) >1
\end{align*}
(where as before $D_2^-$ is the Beurling density of a subset in $\bR^2$, and
$D^-$  is the Beurling density of a subset in $\bR$).

By Theorem \ref{tm_uniq}, $\cG(\varphi, (-\Lambda ) \times \bZ)$ spans a  dense
subspace of  $W_0(\bR)$. Theorem~\ref{thm_uni_sis} now implies that
$x+\Lambda $ is a uniqueness set for $V^\infty (\varphi )$ for all $x\in
\bR $.

(b) If $D^-(\Lambda)>1$ and, in addition,
$\Lambda$ is separated, then,  by Theorem \ref{th_ly_seip},  $\cG(\varphi,(-\Lambda ) \times \bZ)$ is a frame
for $L^2(\bR)$. Hence $\Lambda $ is a sampling set for $\sisp^2(\varphi )$ by
Theorem~\ref{thm:RSLambda}.

Part (c) follows similarly, this time we invoke Theorem \ref{th_interp}.
\end{proof}

\begin{remark}
\normalfont{
(i) The functions in $V^2(\varphi)$ and $V^\infty (\varphi )$ are
certain entire functions of order $2$. It would be interesting to find
a direct proof without the detour to Gabor frames in Theorem~\ref{tm_uniq}.

(ii)
The  critical  case $\Lambda= x+ \bZ$ (hence $D^-(\Lambda)=1$) for
arbitrary generator $g\in W_0$ is
easy to understand with characterization (d) of
Theorem~\ref{thm:RSLambda}.
The quasi-periodicity of the Zak transform implies that $Zg(x+k,\xi) =
e^{-2\pi i k\xi } Zg(x,\xi )$. It follows that $x+\Lambda $ is a
sampling set, \fif\ the set $\{e^{2\pi i k\xi } Zg(x,\xi ): k\in \bZ
  \}$ is a frame for $L^2(\bT, d\xi )$. This is the case, \fif\ $Zg(x,\xi )
  \neq 0$ for all $\xi \in\bR $. However, since $g\in W_0$ has a
  continuous Zak transform, there exists at least one point $(x_0,\xi_0)$ such that $Zg(x_0,\xi _0)=0$ (see e.g.
\cite{bhw95}).
Consequently, $x_0+\Lambda $ is not a sampling set. This case was investigated and understood early on by
Janssen~\cite{Jan93}, Walter~\cite{walter94}, and  Baxter and
Sivakumar ~\cite{BaxterSiva1996}.
}
\end{remark}

\section{Nonuniform sampling with totally positive generator}
\label{sec_samp_tp}

 Next we study the sampling problem in \sis s with a totally positive
 function of Gaussian type as a generator. As explained in the
 introduction, these are all (totally positive) functions $g$ whose Fourier
 transform factors as
\begin{equation}\label{eq:ch1a}
  \hat g(\xi)= \prod_{j=1}^n (1+2\pi i\delta_j\xi)^{-1} \, e^{-c \xi
    ^2},\qquad \delta_1,\ldots,\delta_n\in\bR, c >0 \, .
\end{equation}
By Schoenberg's characterization~\cite{Schoenb1947a}, every $g$ of this form
is totally positive, and every totally positive function
in $L^1(\bR)$ is a limit of \tp\  functions of Gaussian type.
In addition, we note that
$g$ has stable integer shifts. This can be easily
verified using condition (d) of Theorem \ref{thm:JiaMicchelli} and the
fact that the Fourier transform of $g$ in \eqref{eq:ch1a}  has no real zeros.

Next, we denote the zero set of a continuous function $f$ by
\begin{align*}
N_f:=\{x\in \bR: f(x)=0\} \, .
\end{align*}
The identity operator is denoted by $I$ as usual.

The following lemma is the crucial insight that allows us to go from
the Gaussian  to totally  positive functions of Gaussian type as
generators.

\begin{lemma}\label{Rolle}
 Let $f\in C^1(\bR)$ be real-valued. For $a\in \bR , a \neq 0,$ let
$g=\left(I+a\frac{d}{dx}\right)f$.  Then
\begin{align*}
    D^-(N_{g})\ge D^-(N_f).
\end{align*}
\end{lemma}

\begin{proof}
Note that $I + a \tfrac{d}{dx} = a e^{-x/a} \tfrac{d}{dx} e^{x/a}
$. We define  $h\in C^1(\bR)$ by $h(x)=a e^{ x/a}f(x)$ and note that $N_h = N_f$. Furthermore,
 since
$$ g(x)= \left(I+a\frac{d}{dx}\right)f(x)= e^{- x/a}h'(x) ,
$$
we conclude that $N_g=N_{h'}$. It remains to show that $D^-(N_{h'})\ge D^-(N_h)$.

Let $x\in\bR$, $R>0$, and $F \subseteq N_{h}\cap[x-R,x+R]$ a finite subset.
By Rolle's theorem,
$$
     \#(N_{h'}\cap[x-R,x+R]) \ge \# F-1.
$$
Since this holds for every finite subset $F \subseteq
N_{h}\cap[x-R,x+R]$, it follows that $\#(N_{h'}\cap[x-R,x+R]) \ge
\#(N_{h}\cap[x-R,x+R]) -1$. Therefore,
$D^-(N_{h'})\geq  D^-(N_h)$, and
$$
D^-(N_g) =  D^-(N_{h'})\ge D^-(N_h) = D^-(N_f) \, ,
$$
 as claimed.
\end{proof}

We now state  our main results on sampling, which are  an extended form of
Theorems~\ref{tmch1} and ~\ref{tmvardim} in the introduction.

\begin{tm}\label{thm:SamplingSet}
Let $g$ be a totally positive function of Gaussian type.

\begin{itemize}
 \item[(a)]  If  $\Lambda\subseteq \bR$ has density  $D^-(\Lambda)>1$, then
 $\Lambda$ is a uniqueness set for $\sisp^\infty(g )$.
\item[(b)] If $\Lambda$ is separated and $D^-(\Lambda)>1$, then $\Lambda$ is a sampling
set for $\sisp^2(g)$.
\end{itemize}

\end{tm}
\begin{proof}
Recall that $g$ is real-valued and has stable integers shifts.

(a)
Let $c\in\ell^\infty(\bZ)$ and assume that $f=\ctilg \in \sisp^\infty(g)$
vanishes on $\Lambda$ with $D^-(\Lambda ) >1$. We want to show that $f \equiv 0$.
Note that $f\in C^\infty(\bR)$. Since $g$ is real-valued, we may
assume without loss of generality that  $f$ is also real-valued (by  replacing $c_k$ by $\Re(c_k)$
or $\Im(c_k)$ if necessary).

 The representation of $\hat g$ in \eqref{eq:ch1a} implies  that
\begin{align*}
\prod_{j=1}^n \left(I+\delta_j \frac{d}{dx}\right) g =\varphi \, ,
\end{align*}
where $\widehat{\varphi}(\xi)=e^{-c\xi^2}$. Now set
$$\Phi =\sum_{k\in\bZ} c_k \varphi (\cdot-k)\, .$$
Since $\varphi $, $g$ and their derivatives decay exponentially, we may interchange
summation and differentiation in $f$, and obtain that
$$
\Phi = \prod_{j=1}^n \left(I+\delta_j \frac{d}{dx}\right) f \, .
$$
We are now in a position to apply  Lemma \ref{Rolle} repeatedly and conclude that
$$ D^-(N_{\Phi }) \ge D^-(N_f) \geq D^{-}(\Lambda) >1.$$
By Theorem~\ref{cor:Gauss}, the set $N_{\Phi}$ is a uniqueness set for
$\sisp^\infty(\varphi )$.
Since obviously  $\Phi  \equiv 0$ on $N_{\Phi}$, we conclude that
$\Phi \equiv 0$. Hence $c =  0$ and  therefore  $f \equiv
0$. Thus $\Lambda $ is a uniqueness set for $V^\infty (g)$.

(b) Now assume that $\Lambda$ is separated.
By Theorem \ref{thm:linfinity}, in order to show that
$\Lambda$ is a sampling set for $\sisp^2(g)$, it suffices to show that every $\Gamma \in W_\bZ(\Lambda)$ is a uniqueness set for $\sisp^\infty(g)$. Let $\Gamma \in W_\bZ(\Lambda)$.
Since $\Lambda$ is separated, the weak-limit $\Gamma$ is also separated and it satisfies
$D^{-}(\Gamma) \geq D^-(\Lambda)>1$ (by Lemma \ref{lemma_dens_lim} in the Appendix).
Part (a) now asserts that $\Gamma $ is a uniqueness set for $V^\infty(g)$. Therefore, by Theorem \ref{thm:linfinity},
$\Lambda$ is a sampling set for $\sisp^2(g)$.
\end{proof}
\begin{remark}
\normalfont{
In the above proof, we do not know if the set $N_{\Phi}$ is separated, even if the original set $\Lambda$ is.
Therefore we cannot conclude directly that $N_{\Phi}$ is a sampling
set for $\sisp^2(\varphi )$, but only that $N_\Phi $ is a uniqueness set for
$\sisp^\infty(\varphi )$. In this way, the proof exploits the full strength of Theorems \ref{thm:linfinity}
and \ref{tm_uniq}.}
\end{remark}

Using a clever trick by Janssen and Strohmer~\cite{JS02}, we  obtain
the following generalization of Theorem~\ref{thm:SamplingSet}. Let $g$ be a
totally positive function of Gaussian type and  let $c , d\in \ell ^1(\bZ )$ be two sequences with Fourier
series $\hat c(\xi ) = \sum _{k\in \bZ } c_k e^{2\pi i k\xi }$ and
$\hat d$,   such that
$$
\inf _{\xi \in \bR } |\hat c (\xi )| >0 \quad \text{ and } \inf
_{\xi \in \bR } |\hat d (\xi )|  >0 \, .
$$
Now  set
\begin{equation}
  \label{eq:ch5}
  \gamma = \sum _{k,l \in \bZ } c_k d_l T_k M_l g \, .
\end{equation}

\begin{cor} \label{cortrick}
Let $\gamma $ be defined as in \eqref{eq:ch5} and $\Lambda \subseteq
\bR $ be a separated set with density $D^-(\Lambda ) >1$. Then
$\Lambda $ is a sampling set for $\sisp^2(\gamma )$.

In particular, if $\gamma(x)=\mathrm{sech}(\nu x)  = 2 (e^{\nu x} + e^{-\nu x})\inv $ and
$D^-(\Lambda ) >1$, then $\Lambda $ is  a sampling set for $\sisp^2(\gamma )$.
\end{cor}

\begin{proof}
Since $c,d\in \ell ^1(\bZ )$, it is easy to see that $\gamma \in
W_0$. Therefore the characterization in Theorem~\ref{thm:linfinity} is applicable.
From the quasi-periodicity of the Zak transform we
obtain that
\begin{align}
\label{eq_fact}
Z\gamma (x,\xi)=  \hat c (-\xi ) \hat d(x) \, Zg(x,\xi) \, .
\end{align}
Since $|\hat c | , |\hat d| \geq C>0$ and $Zg $ satisfies the
condition~\eqref{eq:gablambda1} of Theorem~\ref{thm:RSLambda}, it
follows that $Z\gamma $ also satisfies
condition~\eqref{eq:gablambda1}. Therefore $\Lambda $ is a sampling set
for $\sisp^2(\gamma )$.

Finally,  Janssen and Strohmer \cite{JS02} showed that  the hyperbolic
secant ${\rm sech}(\nu x)$ possesses a representation \eqref{eq:ch5} with
respect to the Gaussian for all $\nu >0$, whence the statement
follows.
\end{proof}

\section{Gabor frames with totally positive generator}
\label{sec_gab}

We now use the relation between sampling sets for \sis s and Gabor
frames to derive a new result about Gabor frames with totally positive
windows.

Despite intensive research about Gabor frames, the only complete
results cover the Gaussian $\varphi (x) = e^{-cx^2}$, the hyperbolic
secant $g(x) = (e^{cx} + e^{-cx})\inv$, and totally
positive functions of finite type, which are defined by their Fourier
transform as $  \hat g(\xi)= \prod_{j=1}^n (1+2\pi
i\delta_j\xi)^{-1}$, $\delta_1,\ldots,\delta_n\in\bR \setminus \{0\}.$
Our current knowledge can be summarized as follows:

\begin{tm}  \label{tpfin}
Let $g$ be a Gaussian, the hyperbolic secant, or a  totally positive
function of finite type $n\geq 2, n\in\bN$. Then
the Gabor family $\cG(g,\alpha\bZ \times \beta\bZ)$ is a frame for $L^2(\bR)$
if and only if  $\alpha\beta<1$.
\end{tm}

The result goes back to ~\cite{lyub92,seip92,janssen94,JS02}, the case
of totally positive functions of finite type was settled in
\cite{GS13} with a new proof in~\cite{kloos-stoeckler}.
The case $n=1$ corresponds to the one-sided exponential function $h(x)
= e^{-\delta x} \chi _{[0,\infty )}(\delta x)$ and was already settled
by Janssen~\cite{janssen96}: $\cG(h,\alpha\bZ \times \beta\bZ)$ is a frame, \fif\
$\alpha \beta \leq 1$.

The method for the proof in \cite{GS13}
uses techniques from approximation theory, namely the Schoenberg-Whitney
conditions for the invertibility of certain submatrices  of the
pre-Gramian $P_\Lambda (g)$.
Note that the  lattice $\otherset=\alpha \bZ\times \beta\bZ$ has
Beurling density $D^-(\otherset)=(\alpha\beta)^{-1}$, so that
$\alpha\beta<1$ is precisely  the density condition
$D^-(\otherset)>1$.

We extend Theorem~\ref{tpfin} in two ways: we use totally positive functions of
Gaussian type as window functions, and we study semi-regular sets of
the form $\Lambda \times  \beta \bZ $ instead of lattices $\alpha \bZ
\times \beta \bZ $. The following implies Theorem \ref{th_intro_gab}
of the  introduction.

\begin{tm}
Let $g$ be a totally positive function of Gaussian type  and let
$\Lambda\subseteq \bR$ be a separated set.
Then the Gabor family $\cG(g,\Lambda \times \beta\bZ)$ is a frame for $L^2(\bR)$
if and only if  $0<\beta<D^-(\Lambda)$.
\end{tm}

\begin{proof}
We  use the scaling invariance of the class of totally positive
functions of Gaussian type: if $g$  is of the form  \eqref{eq:ch1a},
then the function $g_\beta(x)=g( x/\beta)$ is of the same form
\eqref{eq:ch1a} with different parameters and a positive normalization
factor. By the usual dilation argument $\cG(g,\Lambda \times \beta\bZ)$  is a frame
if and only if $\cG(g_\beta,\beta\Lambda \times \bZ)$ is a frame. Moreover, simple geometric facts yield
$$  D_2^-(\Lambda\times \beta\bZ) = D^-(\beta \bZ ) \, D^-(\Lambda )  = \frac{D^-(\Lambda)}{\beta}.
$$
The combination of  Theorems \ref{thm:RSLambda} and
\ref{thm:SamplingSet} implies that  the Gabor family $\cG(g_\beta,\beta\Lambda \times \bZ)$
is a frame for $L^2(\bR)$, whenever $0<\beta<D^-(\Lambda)$.

The necessity of the condition $D^-(\Lambda \times \beta \bZ ) >1$
follows from recent nonuniform versions of the Balian-Low theorem
\cite{AFK14,grorro15}, which are applicable because $g$ satisfies the
decay and integrability condition
\begin{align*}
\int_{\bR^2} \abs{\ip{g}{M_\xi T_x g}} dx d\xi < \infty.
\end{align*}
(In the standard jargon of time-frequency analysis, this condition
means that $g$ belongs to the modulation space
$M^1(\bR)$.)
This completes the proof.
\end{proof}

\section{Appendix}
\subsection{Density, separation and weak convergence}
In the following we need the dual space $W_0^*$.  We can identify the dual space of $W_0$ with the
space of complex-valued Radon measures $\mu $ such that
\begin{align*}
\sup_{x\in\bR} \norm{\mu}_{C^*([x,x+1])} = \sup_{x\in\bR} \abs{\mu}([x,x+1]) < \infty.
\end{align*}
In the usual notation of amalgam spaces the dual is  $W_0^* =
W(\mathcal{M},\ell^\infty )$~\cite{fe83,he03}.
For completeness, we prove the following folklore lemma.
\begin{lemma}
\label{lemma_dens_lim}
Let $\Lambda \subseteq \bR$ be separated and $\Gamma \in \W(\Lambda)$. Then
$\Gamma$ is separated and $D^{-}(\Gamma) \geq D^{-}(\Lambda)$.
\end{lemma}
\begin{proof}
Let $\{x_n:n\geq 1\} \subseteq \bR$ be such that $\Lambda + x_n \weakconv \Gamma$.
It is easy to see that $\Gamma$ is separated because $\Lambda$ is, and that the measures
$\mu_n := \sum_{\lambda \in \Lambda} \delta_{\lambda+x_n}$ are in
$W_0^*$ and  converge to
$\mu := \sum_{\gamma \in \Gamma} \delta_{\gamma}$ in the $\sigma(W^*_0(\bR), W_0(\bR))$ topology.
(Here  it is important that $\Lambda$ is separated; See Remark \ref{rem_mu_sep}.)
Let $\varepsilon \in (0,1/2)$ and $\eta \in C^\infty(\bR)$ be such that $0 \leq \eta \leq 1$, $\supp(\eta) \subseteq
[-1,1]$ and
$\eta \equiv 1$ on $[-1+\varepsilon,1-\varepsilon]$.

There is nothing to prove if $D^-(\Lambda ) = 0$.  If $D^-(\Lambda )
>0$,   choose $\rho \in \bR$  such that $0<\rho<D^{-}(\Lambda)$. Then
there exists $r_0>0$ such that for all $r \geq r_0$ and all
$x \in \bR$ we have  $\#(\Lambda \cap [x-r,x+r]) \geq 2r\rho$. For $x \in \bR$ and $r \geq 2r_0$,
\begin{align*}
\#(\Gamma \cap [x-r,x+r]) &\geq \int_{\bR} \eta\big(\frac{y-x}{r}\big) d\mu(y)
= \lim_{n\to \infty } \int_{\bR} \eta\big(\frac{y-x}{r}\big) d\mu_n(y)
\\
&\geq \liminf_{n\to \infty } \#((\Lambda + x_n) \cap [x-r(1-\varepsilon),x+r(1-\varepsilon)])
\\
&= \liminf_{n \to \infty } \#(\Lambda  \cap
[x-x_n- r(1-\varepsilon),x-x_n +r(1-\varepsilon)]) \\
&\geq 2r\rho(1-\varepsilon).
\end{align*}
Hence $D^{-}(\Gamma) \geq \rho(1-\varepsilon)$. As this holds for
arbitrary
$\varepsilon > 0$ and $\rho < D^{-}(\Lambda)$,
we conclude that $D^{-}(\Gamma) \geq D^{-}(\Lambda)$.
\end{proof}
\begin{remark}
\label{rem_mu_sep}
\normalfont{
Lemma \ref{lemma_dens_lim} is false  for non-separated sets
$\Lambda$. For example, if
$\Lambda = \{k,k+2^{-|k|}: k \in \bZ\}$, then $D^-(\Lambda ) = 2$, but
$\Lambda - n \weakconv \bZ$, as $n \longrightarrow \infty$ with $n \in \bN$,
and
$D^-(\bZ ) = 1$.
In this case, the measure $\mu$ in the proof is $\mu=\sum_{k\in \bZ} 2\delta_k$.
}
\end{remark}

\subsection{Gabor and sampling: postponed proofs}
\label{sec_pp}

\begin{proof}[Proofs of Theorems \ref{thm:RSLambda} and \ref{th_interp}]
\mbox{}

\noindent {\bf Step 1}. \emph{Relation between  the relevant operators}.

The spanning properties of the Gabor system
$\mathcal{G}(g,(-\Lambda)\times\bZ)$ on $L^2(\bR)$
are encoded in the spectrum of the synthesis operator
\begin{align*}
  C: \, \ell^2(\Lambda\times\bZ) &\longrightarrow L^2(\bR)
  \\
  \left(c_{\lambda,k} \right)_{\lambda \in \Lambda, k \in \bZ}
  &\mapsto \sum_{\lambda\in \Lambda} \sum_{k \in \bZ} c_{\lambda, k} e^{2 \pi i k \cdot} g(\cdot+\lambda).
\end{align*}
Indeed, $\mathcal{G}(g,(-\Lambda)\times\bZ)$ is a Riesz sequence if
and only if $C$ is bounded below, and
$\mathcal{G}(g,(-\Lambda)\times\bZ)$ is a frame if and only if $C$ is
surjective,  if and only if
$C^*$ is bounded below.

Similarly, the properties of
the sets $\Lambda+x$ as sampling and interpolating
sets for $\sisp ^2(g)$ are determined by the operators $P_{\Lambda +x}(g):
\ell^2(\bZ) \to \ell^2(\Lambda)$ represented by the matrices
$P_{\Lambda + x}(g) =
(g(x+\lambda-k))_{\lambda \in \Lambda ,k\in \bZ}$, or, equivalently, by the spectral properties of their
Banach adjoints
\begin{align*}
  P_{\Lambda +x}(g)':\, \ell^2(\Lambda) &\longrightarrow \ell^2(\bZ)
  \\
  (c_\lambda)_{\lambda \in \Lambda} &\mapsto
  \sum_{\lambda \in \Lambda} c_\lambda g(x+\lambda-k).
\end{align*}
More precisely, according to Lemma \ref{lemma_samp_ops},  $\Lambda+x$
is a sampling set for $\sisp ^2(g)$ if and only if $P_{\Lambda +x}(g)$ is bounded
below and it is an
interpolating set if and only if $P_{\Lambda +x}(g)'$ is bounded below.

It remains to  clarify the relation between the operators $C$ and
$P_{\Lambda + x}(g)$.
We consider the following unitary maps between Hilbert spaces. Let $I
= [0,1)$ and
$U: \ell^2(\Lambda\times\bZ) \to L^2(I,\ell^2(\Lambda))$ be given
by $U\big( (a_{\lambda, k})_{(\lambda,k)\in\Lambda \times \bZ }) = \big( m_\lambda
(x))_{\lambda \in \Lambda }$ with $m_\lambda (x) = \sum _{k\in \bZ }
a_{\lambda, k} e^{2\pi i kx}$ and $V: L^2(\bR) \to L^2(I, \ell^2(\bZ))$
be given by $Vf(x) = \left(f(x-k)\right)_{k\in\bZ})$ for $x\in I$.
Now consider the map
\begin{align*}
  \widehat C: \,L^2(I,\ell^2(\Lambda )) &\longrightarrow L^2(I, \ell^2(\bZ))
  \\
  \widehat C  \Big(x \mapsto (m_\lambda(x))_{\lambda \in \Lambda} \Big)
  &=  \left(
  x \mapsto \left(\sum_{\lambda\in \Lambda} m_{\lambda}(x) g(x+\lambda-k) \right)_{k\in\bZ} \right).
\end{align*}
In technical jargon, $\widehat C$ is the \emph{direct integral}
\begin{align*}
\widehat C = \int^\oplus_I P_{\Lambda +x}(g)' \, dx.
\end{align*}
These definitions imply that
\begin{equation}
  \label{eq:ch89}
  VC=\widehat C U \, ,
\end{equation}
and thus $C$ and $\widehat C$ have the same spectral properties. It is
a standard fact about direct integrals  that $\widehat C$ is bounded
below if and only if
$\essinf_x \sigma_{\min} (P_{\Lambda +x}(g)') >0$
- where $\sigma_{\min} (T) :=
\inf \{\norm{Tx}:\norm{x}=1\}$.  Similarly,  $\widehat C^*$ is bounded below
if and only if $\essinf_x \sigma_{\min} (P_{\Lambda +x}(\overline{g}) )=
\essinf_x \sigma_{\min} (P_{\Lambda +x}(g) ) >0$ (see \cite[Chapter 1, Part II]{dixmier81}
for a general reference on direct integrals or \cite{janssen95,
  mystery} for direct computations in related contexts).

\noindent {\bf Step 2}. \emph{Conclusions}.
We only argue for Theorem \ref{thm:RSLambda}, Theorem \ref{th_interp}
follows in the same manner.

Assumption (c) says that  $P_{\Lambda +x}(g)$ is bounded below with
uniform constants. The direct integral decomposition~\eqref{eq:ch89} now
implies  the equivalence between (a) and (c). The implication (b)
$\Rightarrow $ (c)  follows from the fact that
the map $x \mapsto P_{\Lambda +x}(g)$ is continuous from $\bR $ to
$\mathcal{B}(\ell^2(\bZ),  \ell^2(\Lambda))$ with respect to  the operator norm. The
continuity is a special case of various jitter error bounds in the sampling literature, see
e.g. \cite[Lemma 2.2 and Theorem 2.3]{femoro08}, and holds precisely
when  $g \in W_0$.   Finally, the equivalence of (c) and (d) is a straightforward
application of Parseval's identity with  $h(\xi)= \sum_{k\in\bZ} c_k e^{2\pi i k\xi}$ and  the definition
of the Zak transform.
\end{proof}
\begin{remark}
\normalfont{
Common signal processing proofs of Theorems \ref{thm:RSLambda} and
\ref{th_interp} for the uniform case $\Lambda = \alpha \bZ $
use the Poisson summation formula  and are therefore not applicable to
non-lattice sets $\Lambda$. The fact that these theorems are true in the stated generality seems to have gone
so far unnoticed.
}
\end{remark}

\begin{proof}[Proof of Theorem \ref{tm_uniq}]
Without loss of generality we consider the standard Gaussian $\varphi(x) := e^{-\pi x^2}$. The
general case $e^{-c x^2}$ follows then by applying the following change of variables:
$f(x) \mapsto f(a x)$, $\otherset \mapsto \left[ \begin{smallmatrix} a^{-1} & 0\\ 0& a \end{smallmatrix}
\right] \cdot \otherset$, $a=\sqrt{c/\pi}$, which preserves the density of $\otherset$.

In order to show that $\cG(\varphi,\otherset)$ spans a dense subspace of $W_0$, we show that the only linear functional
on $W_0$ that annihilates $\cG(\varphi,\otherset)$ is the zero
functional. In the following
\begin{align}
\label{eq_stft}
V_\varphi \mu(x,\xi) := \int_{\bR} \varphi (t-x) e^{- 2\pi i \xi t}
d\mu(t)  \qquad (x,\xi ) \in \bR ^2 \,
\end{align}
denotes the  the short-time Fourier
transform of a (Radon) measure $\mu$ (with respect to the Gaussian
window $\varphi$).

Now  suppose that $\mu \in W_0^*$ is a measure such that $V_\varphi
\mu (\nu_1 , \nu_2 ) = 0$
for all $(\nu _1 , \nu _2 ) \in \otherset$. We want to show that $\mu \equiv
0$.  By the properties of the short-time Fourier transform $V_\varphi
\mu $ is bounded for $\mu \in W_0^*$ and
satisfies the identity $V_\varphi  \mu (x,- \xi) = F(z) e^{-\pi i x \xi}
e^{-\tfrac{\pi}{2}\abs{z}^2}$, where $z=x+i\xi \in \bC$ and $F$ is the
Bargmann transform  of $\mu$. Hence  $F$
is
analytic and satisfies
\begin{align*}
\norm{F}_{\mathcal{F}^\infty} := \sup_{z\in\bC} \left| F(z) e^{-\tfrac{\pi}{2}\abs{z}^2} \right| <
\infty.
\end{align*}
Moreover, $F \equiv 0$ if and only if $\mu=0$. (See for example
\cite{ba61}, \cite{folland89}, or  \cite[Section 3.4]{book}.)

With these identifications, we reduce the problem to the
following. Let $F: \bC \to \bC$ be an analytic function
such that $F \equiv 0$ on $\overline{\otherset}=\{x-i\xi: (x,\xi) \in \otherset\}$ and $\norm{F}_{\mathcal{F}^\infty} <
\infty$.
We want to show that $F \equiv 0$. Suppose not.
We can assume that $\norm{F}_{\mathcal{F}^\infty} = 1$ and that
$F(0) \not=
0$. (If $F(0)=0$, consider $F(z)/z^k$ for a suitable power and
normalize.)  Then $F:\bC \to \bC$ is an entire function that satisfies
\begin{align}
\label{eq_aa}
\abs{F(z)} \leq e^{\tfrac{\pi}{2}\abs{z}^2}.
\end{align}
Let $N(r)$ be the number of zeros of $F$ inside the open disk $D_r$  counted with multiplicities.
By Jensen's formula,
\begin{align*}
\int_0^r \frac{N(t)}{t} dt = \frac{1}{2\pi} \int_0^{2\pi} \log\abs{F(r e^{i\theta})} d\theta - \log\abs{F(0)}.
\end{align*}
By \eqref{eq_aa}, $\log\abs{F(r e^{i\theta})} \leq \tfrac{\pi}{2}r^2$. Therefore,
\begin{align}
\label{eq_bb}
\int_0^r \frac{N(t)}{t} dt \leq \frac{\pi}{2}r^2 - \log\abs{F(0)}.
\end{align}
On the other hand, since
$D^{-}\left(\overline{\otherset}\right)=D^{-}(\otherset)>1$, we can
choose $\nu >1$ and $r_0>0$
such that,
for $r \geq
r_0$, $N(r) \geq \nu \pi
r^2$. As a consequence,
\begin{align*}
\int_0^r \frac{N(t)}{t} dt \geq \nu  \frac{\pi}{2}r^2 -C,
\end{align*}
for some constant $C>0$. Since $\nu >1$, this contradicts \eqref{eq_bb} for $r \gg 0$, and hence $F \equiv 0$, as
desired.
\end{proof}

\begin{remark}
\normalfont{
The proof of Theorem \ref{tm_uniq} actually shows that $\cG(g,\otherset)$ spans a dense subspace of the space
$M^1(\bR)$,
which consists of functions with integrable short-time Fourier transform. $M^1(\bR)$ embeds continuously into
$W_0(\bR)$.
}
\end{remark}

\subsection{Beurling-type characterization of sampling: postponed proof}
\begin{proof}[Proof of Theorem \ref{thm:linfinity}]
$(a) \Leftrightarrow (b)$ We consider the matrix $P_\Lambda(g)=(g(\lambda-k))_{\lambda \in \Lambda,k\in\bZ}$ as an
operator $\ell^p(\bZ) \to
\ell^p(\Lambda)$. To show that $(a) \Leftrightarrow (b)$, we invoke
the following non-commutative version of Wiener's Lemma taken from
\cite[Proposition 8.1]{grorro15}:
\begin{prop}
\label{prop_loc}
Let $\Lambda$ and $\Gamma$ be relatively separated subsets of $\bR$
and  $A \in \bC^{\Lambda \times \Gamma}$ be a
matrix such that
\begin{align*}
\abs{A_{\lambda,\gamma}} \leq  \Theta (\lambda-  \gamma)
\qquad \lambda \in \Lambda, \gamma \in \Gamma,
\qquad \mbox{for some } \Theta  \in W_0 \, .
\end{align*}
Then $A$ is bounded below on some $\ell ^{p_0}(\Gamma )$ $1\leq p_0
\leq \infty $, i.e.,  $
  \norm{Ac}_{p_0} \geq C_0 \norm{c}_{p_0}$  for all $ c\in \ell
  ^{p_0} (\Gamma)$, \fif\ $A$ is bounded below on \emph{all} $\ell ^p
  (\Gamma )$, $1\leq p\leq \infty $.
\end{prop}
Since $g \in W_0$, the envelope for $P_\Lambda (g)$ can be chosen to
be $\Theta (x) = |g(x)|$. Then $P_{\Lambda}(g)$ is bounded below on
some $\ell ^{p_0}$, \fif\ it is bounded below on all $\ell ^p$. This
gives the desired conclusion.

 The
statement in \cite{grorro15}  actually covers more general  matrices
that are  concentrated away from a collection of lines; with the notation of
\cite[Proposition 8.1]{grorro15} we apply it with $G=\{1\}$.  For related versions of Wiener's Lemma we refer to Sj\"ostrand's
fundamental work~\cite{sj95} and \cite{ABK08,bas90,Sun07a}.

 $(b) \Rightarrow (c)$  Suppose that $\Lambda$ is a
sampling set for $\sisp^\infty(g)$ and let $\Gamma \in W_{\bZ
}(\Lambda )$,  so that there exist $k_n \in
\bZ$ such that $\Lambda - k_n \weakconv \Gamma$.  Since $\Lambda$ is a sampling set, the operator
$P_\Lambda (g): \ell^\infty(\bZ) \to \ell^\infty(\Lambda)$ is bounded below, so its (pre) adjoint
$P_\Lambda(g)^\prime: \ell^1(\Lambda) \to \ell^1(\bZ)$, represented by the matrix
\begin{align*}
P_\Lambda (g)^\prime = (g(\lambda-k))_{k\in\bZ, \lambda \in \Lambda},
\end{align*}
is onto  by the closed range theorem. This means that for every  sequence $c \in \ell^1(\bZ)$
there exists an
$a \in \ell^1(\Lambda)$ such that
\begin{align}
\label{eq_a}
c_k = \sum_{\lambda \in \Lambda} a_\lambda g(\lambda-k),
\qquad k \in \bZ,
\end{align}
and $\norm{a}_1 \lesssim \norm{c}_1$. For  $c \in \ell^1(\bZ)$ fixed, we apply this observation
to every sequence $(c_{k-k_n})_{k \in \bZ}$ and  find a corresponding
sequence $a^n \in \ell^1(\Lambda)$ with
$\norm{a^n}_1 \lesssim \norm{c}_1$ and
\begin{align}
\label{eq_b}
c_{k-k_n} = \sum_{\lambda \in \Lambda} a^n_\lambda g(\lambda-k),
\qquad k \in \bZ.
\end{align}
Consider the shifted sequence $b^n \in \ell^1(\Lambda - k_n)$ defined by $b^n_{\lambda - k_n} = a^n_\lambda$. We can
rewrite \eqref{eq_b} as
\begin{align}
\label{eq_c}
c_{k} = \sum_{\nu \in \Lambda-k_n} b^n_\nu g(\nu-k),
\qquad k \in \bZ.
\end{align}
Consider the measure $\mu_n := \sum_{\nu \in \Lambda-k_n} b^n_\nu
\delta_\nu$. Then $\mu _n$ is a bounded measure, i.e.,  $\mu \in
C^*_0(\bR )$, and  $\norm{\mu_n}_{C^*_0(\bR)} = \norm{b^n}_1 = \norm{a^n}_1 \lesssim \norm{c}_1$.
By passing to a subsequence, we may choose a measure $\mu \in C^*_0(\bR)$ such that
$\mu_n \longrightarrow \mu$ in the weak$^*$ topology  $\sigma(C^*_0(\bR),C_0(\bR))$.
Since $\Lambda-k_n \weakconv \Gamma$, it follows that $\Gamma$ is relatively separated and
$\supp(\mu) \subseteq \Gamma$ (see \cite[Lemmas 4.3 and 4.5]{grorro15}).
Hence $\mu = \sum_{\gamma \in \Gamma} b_\gamma \delta_\gamma$, for some sequence $b$.
Moreover $\norm{b}_1 = \norm{\mu}_{C^*_0(\bR)} \lesssim \liminf_n \norm{\mu_n}_{C^*_0(\bR)} \lesssim 1$.
Finally, for each $k \in \bZ$,
\begin{align*}
\sum_{\gamma \in \Gamma} b_\gamma g(\gamma-k)
&=\int_{\bR} g(x-k)\,  d\mu(x)
=\lim_n \int_{\bR} g(x-k) d\mu_n(x)
\\
&=\lim_n \sum_{\nu \in \Lambda-k_n} b^n_\nu g(\nu-k)=c_k.
\end{align*}
Since $c \in \ell^1(\bZ)$ was arbitrary, this shows that
$P_\Gamma (g)^\prime :\ell^1(\Gamma) \to \ell^1(\bZ)$ is onto
and therefore  $P_\Gamma (g):\ell^\infty(\bZ) \to \ell^\infty(\Gamma)$ is
bounded below. Hence $\Gamma$ is a sampling set for $\sisp^\infty(g)$.

The implication $(c) \Rightarrow (d)$ is obvious, the equivalence
$(d) \Leftrightarrow (e)$ is just language, and the equivalence $(e)
\Leftrightarrow (f)$ is standard functional analysis.

$(d) \Rightarrow (b)$  Assume  on the contrary that $\Lambda$ is
not a sampling set for $p=\infty$. Then  there
exist sequences $c^n \in \ell^\infty(\bZ)$ with $\norm{c^n}_\infty =1$
such that
\begin{align}
\label{eq_d}
\sup_{\lambda \in \Lambda} \abs{\sum_{k \in \bZ} c^n_k g(\lambda-k)} \rightarrow 0,
\qquad \mbox{ as } n \rightarrow \infty.
\end{align}
Since $\norm{c^n}_\infty = 1$, there exist $k_n \in \bZ$ such that $\abs{c^n_{k_n}} \geq 1/2$. Let
$d^n \in \ell^\infty(\bZ)$ be the shifted sequence $d^n_{k} :=
c^n_{k+k_n}$. By passing to a  subsequence
we may assume that (i) $d^n \rightarrow d\in \ell ^\infty (\bZ )$
in the $\sigma(\ell^\infty,\ell^1)$ topology,
and (ii) $\Lambda - k_n \weakconv \Gamma$  for some (relatively separated) set $\Gamma \in \WZ(\Lambda)$
(see \cite[Lemma 4.5]{grorro15}). Let $f := \sum_{k \in \bZ} d_k g(\cdot-k) \in \sisp^\infty(g)$.
Since $c^n_{k_n}=d^n_0 \longrightarrow d_0$, we must have $d_0
\not=0$, and therefore $f \not\equiv 0$.

We next  show that  $f \equiv 0$ on $\Gamma$. Given $\gamma \in \Gamma$, there exists a sequence
$\{\lambda_n: n \geq 1\} \subseteq \Lambda$ such that $\lambda_n - k_n \longrightarrow \gamma$. Since
$g \in W_0$, it follows  that $\left( g(\lambda_n - k_n-k) \right)_{k \in \bZ}
\longrightarrow \left( g(\gamma-k) \right)_{k \in \bZ}$ in $\ell^1(\bZ)$. Since
$d^n \longrightarrow d$ in $\sigma(\ell^\infty,\ell^1)$,
we conclude from \eqref{eq_d} that
\begin{align*}
\abs{f(\gamma)}=
\big|\sum_{k \in \bZ} d_k g(\gamma-k) \big|
= \lim_n \,
\big| \sum_{k \in \bZ} d_k^n g(\lambda_n-k_n-k) \big|
= \lim_n
\big| \sum_{k \in \bZ} c^n_k g(\lambda_n-k) \big| = 0.
\end{align*}
Hence $f \equiv 0$ on $\Gamma$, which  contradicts  (d). This completes the proof.
\end{proof}

\end{document}